\documentclass{amsart}

\usepackage[spanish ,english]{babel}
\usepackage[latin1]{inputenc}
\usepackage{dsfont}
\usepackage{graphicx}
\usepackage{color} 

\newtheorem{teo}{Theorem}[section]
\newtheorem{lema}[teo]{Lemma}
\newtheorem{remark}[teo]{Remark}
\newtheorem{corolario}[teo]{Corollary}
\newtheorem{prop}[teo]{Proposition}
\newtheorem{definition}{Definition}
\usepackage{hyperref}

\newcommand{\ind}[1]{\mathds{1}_{#1}}

\def\bibname{Referencias}
\renewenvironment{thebibliography}[1]
     {\footnotesize \section*{\normalsize \bibname}%
      \list{[\arabic{enumi}]}{\settowidth\labelwidth{[#1]}
            \leftmargin\labelwidth
            \advance\leftmargin\labelsep
            \usecounter{enumi}}
            }

\newcommand{\supp}{\mathop{\rm supp}}
\newcommand{\dist}{\mathop{\rm dist}}
\newcommand{\sign}{\mathop{\rm sign}}
\newcommand{\proj}[1]{\mathcal{P}#1}
\newcommand{\suppD}{\mathop{\rm supp}_{\mathcal{D}'}}
\newcommand{\ds}{\,\mathrm{d}s}
\newcommand{\dt}{\,\mathrm{d}t}
\newcommand{\dx}{\,\mathrm{d}x}
\newcommand{\R}{\mathbb{R}}
\renewcommand{\L}{\mathrm{L}}
\newcommand{\C}{\mathrm{C}}

\begin{document}


\title{A nonlocal two phase Stefan problem}

\author{Emmanuel Chasseigne$^\dagger$}
\thanks{$^\dagger$ Laboratoire de Math\'ematiques et Physique Th\'eorique, U. F. Rabelais, Parc de Grandmont, 37200 Tours, France
email: emmanuel.chasseigne@lmpt.univ-tours.fr}
\author {Silvia Sastre-Gómez$^\ddagger$}\thanks{$^\ddagger$ Departamento de Matemática Aplicada, U. Complutense de Madrid. 
Partially supported by the FPU grant from the Spanish 
Ministerio de Educación and the projects MTM2009-07540, UCM-CAM, Grupo de Investigación CADEDIF. email: silviasastre@mat.ucm.es}

\maketitle

\begin{abstract} We study a nonlocal version of the two-phase Stefan
  problem, which models a phase transition problem between two distinct
  phases evolving to distinct heat equations. Mathematically speaking,
  this consists in deriving a theory for sign-changing solutions of the
  equation, $u_t=J\ast v -v $, $v=\Gamma(u)$, where the
  monotone graph is given by $\Gamma(s)=\sign(s)(|s|-1)_+$.
  We give general results of existence, uniqueness and
  comparison, in the spirit of \cite{BCQ2012}. Then we focus on the
  study of the asymptotic behaviour for sign-changing solutions, which
  present challenging difficulties due to the non-monotone evolution of
  each phase.
\end{abstract}

\footnotetext{
	   \\
\textit{keywords:} Stefan problem, mushy regions, nonlocal equations, degenerate parabolic equations.}


\section{Introduction}

The aim of this paper is to study the following nonlocal version of the
two-phase Stefan problem in $\mathbb{R}^N$ 
\begin{equation}\label{eq1} \left\{ \begin{array}{ll} u_t=J\ast v-v, &
  \mbox{where } v=\Gamma(u),\\ u(\cdot,0)=f, & \end{array} \right.
\end{equation} 
where $J$ is a smooth nonnegative convolution kernel, $u$ is the enthalpy and
$\Gamma(u)=\sign(u)\big(|u|-1\big)_+$ (see below more precise
assumptions and explanations). We study this nonlocal equation in the
spirit of \cite{BCQ2012}, but for sign-changing solutions, which
presents very challenging difficulties concerning the asymptotic
behaviour. 

\noindent\textsc{The two-phase Stefan problem --} In general, the Stefan
problem is a non-linear and  moving boundary problem which aims to
describe the temperature and enthalpy distribution in a phase
transition between several states. The history of the problem goes back
to Lam\'e and Clapeyron \cite{Lame}, and afterwards \cite{Stefan}.  For
the local model can be seen e.g. the monographs \cite{Chalmers} and
\cite{Woodruff} for the phenomenology and modeling; \cite{Crank},
\cite{Meirmanov}, \cite{Rubinstein} and \cite{Visintin} for the
mathematical aspects of the model. 

The main model uses a local equation under the form $u_t=\Delta v$,
$v=\Gamma(u)$ but recently, a nonlocal version of the one-phase Stefan
problem was introduced in  \cite{BCQ2012}, which is equivalent to
\eqref{eq1} in the case of nonnegative solutions.

This new mathematical model turns out to be rather interesting from the
physical point of view at an intermediate (mesoscopic) scale, since it
explains for instance the formation and evolution of \textit{mushy
regions} (regions which are in an intermediate state between water and 
ice). We are not going to enter into more details here and refer the 
reader to \cite{BCQ2012} for more information about the model and more 
bibliographical references.

Let us however mention some basic facts: the one-phase problem models
for instance the transition between ice and water: the ``usual'' heat
equation (whether local or nonlocal) governs the evolution in the water
phase while the temperature does not evolve in the ice phase, maintained
at $0^\circ$. The free boundary separating water from ice evolves
according to how the heat contained in water is used to break the ice.   
 
In the two-phase Stefan problem, the temperature can also evolve in the
second phase, modeled by a second heat equation with different
parameters. In this model, the temperature $v=\Gamma(u)$ is the quantity
which identifies the different phases: the region $\{v>0\}$ is the first
phase, $\{v<0\}$ represents the second phase and the intermediate
region, $\{v=0\}$ is where the transition occurs, containing what is
called a \textit{mushy region}.  
  
In all the paper, the function $J$ in equation \eqref{eq1} is assumed to
be continuous, non negative, compactly supported, radially symmetric,
with $ \int_{\mathbb{R}}J=1$\,. We denote by $R_{J}$ the radius of the 
support of $J$: $\supp(J)=B_{R_J}$, where $B_{R_J}$ is the ball 
centered in zero with radius $R_J$. The graph $v=\Gamma(u)$, is defined
generally as follows 
\begin{equation}\label{gamma} \Gamma(u)= \left\{ \begin{array}{ll}
  c_1(u-e_1), 	& 	\mbox{if } u<e_1\\ 0, & \mbox{if } e_1\le u \le
  e_2\\ c_2(u-e_2),	& 	\mbox{if } u>e_2.  \end{array} \right.
\end{equation} 
with $e_1,\,e_2,\,c_1$ and $c_2$ real variables,  that satisfy that
$e_1<0< e_2$ and $c_1,\,c_2>0$ (see Figure \ref{figure} below). After a simple
change of units, we arrive at the graph of equation~\eqref{eq1}:
$\Gamma(u)= \sign(u)\left(|u|-1\right)_+\,,$ where we denote by 
$s_{+}$ the quantity $\max(s,0)$, as is standard and $\sign(s)$ equals 
$-1,\ +1$ or $0$ according to $s<0, \ s>0,$ or $s=0$.

\begin{figure}[htp] \begin{center}
  \includegraphics [height=4cm]{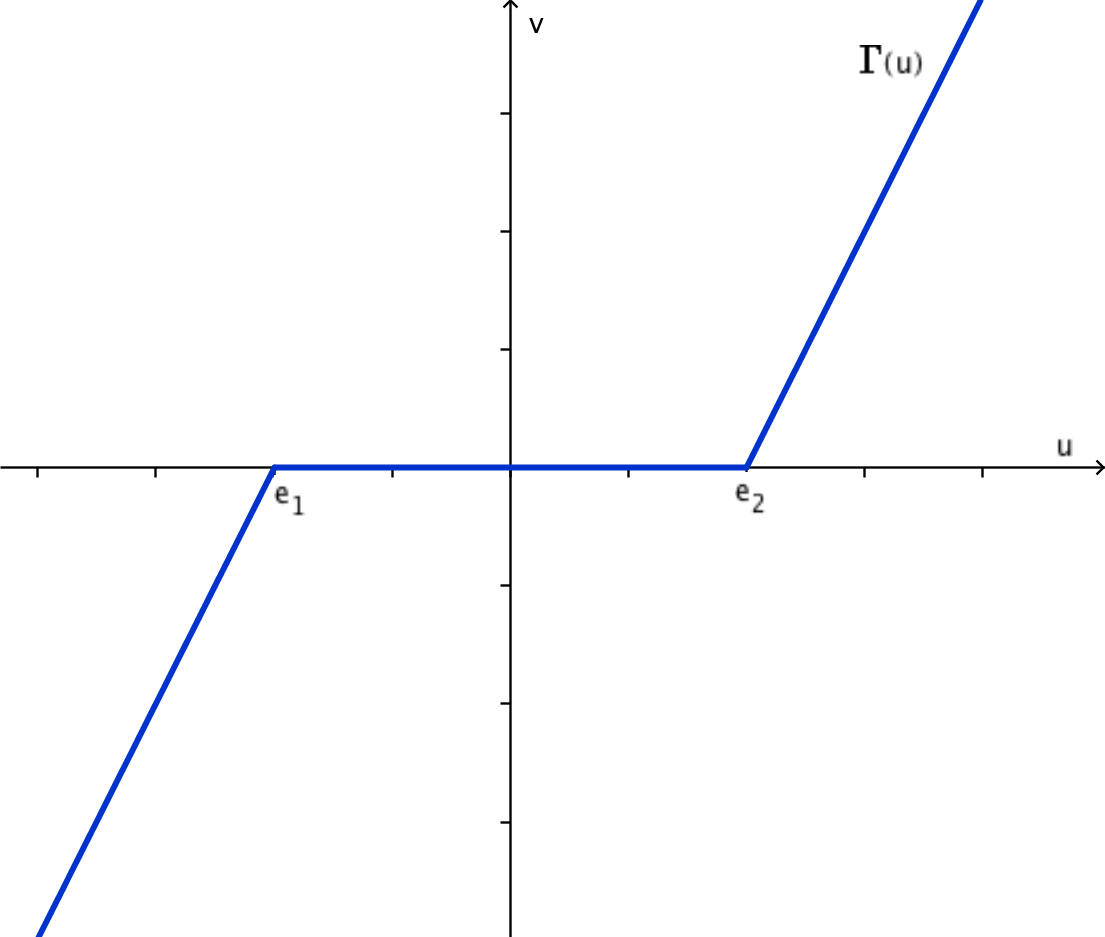}
  \caption{A typical graph $\Gamma$}
  \label{figure}
\end{center} \end{figure}

\noindent\textsc{Asymptotic Behaviour --} 
In \cite{BCQ2012}, the authors proved several qualitative properties for
the nonlocal one-phase Stefan problem. Most of them are also valid in
the two-phase problem, but the asymptotic behaviour is far from being
fully understood when solutions change sign. 

Actually, up to our knowledge, there are no results for the asymptotic
behaviour of sign-changing solutions even in the local two-phase Stefan
problem. The aim of this paper is to try to provide at least some partial
answers. 

Going back to the one-phase Stefan problem, it can be shown that there
exists a projection operator $\proj{}$ which maps any nonnegative
initial data $f$ to $\proj{f}$, which is the unique solution to a
non-local obstacle problem at level one (see \cite[p. 23]{BCQ2012}). 
Then the asymptotic behaviour 
of the solution $u$ starting with $f$ is given by $\proj{f}$. Actually,
this can be done exactly this way if, for example, $f$ is compactly
supported.  Then $\proj{}$ can be extended to all $\L^1$ (the space of 
integrable functions), using a standard closure theory of monotone 
operators.

A key argument in the one-phase Stefan problem is the \textit{retention
property}, which means that once the solution becomes positive at some
point, it remains positive for greater times. In this case, the
interfaces are monotone: the positivity sets (of $u$ and $v$) grow.
With this particular property, the Baiocchi transform gives all necessary 
and sufficient information to derive the asymptotic obstacle
problem (for information about the Baiocchi transform, see \cite{Baiocchi}).
 
In the case of the two-phase Stefan problem, the situation is far more
delicate to handle, due to the fact that sign-changing solutions do not
enjoy a similar retention property in general: a solution can be
positive, but later on it can become negative due to the presence of a
high negative mass nearby. This implies that the Baiocchi transform is
not a relevant variable anymore in general and many arguments fail.

However, we shall study here some situations in which we can still
apply, to some extent, the techniques using the Baiocchi transform and
get the asymptotic behaviour for sign-changing solutions.

\noindent\textsc{Main Results --} we first briefly derive a complete
theory of existence, uniqueness and comparison for the nonlocal
two-phase Stefan problem, which is based essentially on the same ideas
in \cite{BCQ2012}. Then we concentrate on the asymptotic behaviour of
sign-changing solutions. Though we do not provide a complete picture of
the question which appears to be rather difficult, we give some
sufficient conditions which guarantee the identification of the limit.

Namely, we first give in Section~\ref{sect:not.interact} a criterium
which ensures that the positive and negative phases will never interact.
This implies that the asymptotic behaviour is given separately by each
phase, considered as solutions of the one-phase Stefan problem.

Then we study the case when some interaction between the phases can
occur, but only in the mushy zone, $\{|u|<1\}$. In this case we prove
that the asymptotic behaviour can be described by a bi-obstacle problem,
the solution being cut at levels $-1$ and $+1$. We prove that this
obstacle problem has a unique solution in a suitable class, and then we
extend the operator which maps the initial data to the asymptotic limit
to more general data by a standard approximation procedure. Notice that
for the local model, such a result would be rather trivial since the
mushy regions do not evolve. However, here those regions do evolve due
to the nonlocal character of the equation. 

Finally, we give an explicit example when the enthalpy becomes nonnegative 
in finite time even if the initial data is not, so that the asymptotic behaviour 
is driven  by the one-phase Stefan regime.

\noindent\textsc{Notations --} Throughout the paper, we use the 
following notation: $\mathrm{C}(\R^{N};\R)$, or in shorter 
form $\mathrm{C}(\R^{N})$ is the space of continuous functions from 
$\R^{N}$ with values in $\R$. Other spaces we consider:
\begin{itemize}
	\item $\mathrm{BC}\left(\mathbb{R}^N\right)=\{\varphi\in \C\left(\mathbb{R}^N\right):\,\varphi\mbox{ bounded in }\mathbb{R}^N\}$;
	\item $\C_c\left(\mathbb{R}^N\right)=\{\varphi\in \C\left(\mathbb{R}^N\right):\,\varphi\mbox{ compactly supported }\}$;
	\item $\C^{\infty}_c\left(\mathbb{R}^N\right)=\{\varphi\in \C^{\infty}\left(\mathbb{R}^N\right):\,\varphi\mbox{ compactly supported }\}$;
	\item $\C_0\left(\mathbb{R}^N\right)=\{\varphi\in 
	\C\left(\mathbb{R}^N\right):\,\varphi\rightarrow 0\mbox{ as } |x|\rightarrow \infty\}$;
	\item $\L^{1}(\R^{N}) = \{\varphi:\R^{N}\to\R\,,\text{measurable 
	and integrable in }\R^{N}\}$;
	\item $\C\left ([0,\infty); \L^1\left(\mathbb{R}^N\right)\right)$ 
	is the space of functions $t\mapsto u(t)$ wich are continuous in 
	time, with values in $\L^{1}(\R^{N})$ for any $t\geq0$;{}
	\item $\L^{1}\left ([0,T]; \L^1\left(\mathbb{R}^N\right)\right)$ 
	is the space of functions $t\mapsto u(t)$ wich are integrable in 
	time over $[0,T]$, with values in $\L^{1}(\R^{N})$ for any $t\geq0$.
\end{itemize}
Recall that throughout the paper, $J$ is nonnegative, radially symetric, compactly 
supported with $\int J=1$ and $\supp(J)=B_{R_{J}}$. Finally, we denote 
by $s_{+}=\max(s,0)$ and $s_{-}=\max(-s,0)$.


\section{Basic theory of the model}

In this section we will develop the basic theory for the solution of the
two-phase Stefan problem. Some results are already contained in
\cite{BCQ2012} after some obvious adaptation. This is due to the fact 
that for the one-phase Stefan model, $\Gamma(u)=(u-1)_{+}$ while here, 
we deal with a symetric function $\Gamma(u)=\sign(u)(u-1)_{+}$ which 
is very close to the first one.

However, for the sake of completeness, we shall rewrite the proof when 
the adaptation may not be so straightforward, and give the precise 
reference otherwise. 


\subsection {$\L^1$ theory.}

We start with the theory for integrable initial data. In this case the
solution is regarded as a continuous curve in
$\L^1\left(\mathbb{R}^N\right)$.
 
\begin{definition} Let $f\in \L^1\left(\mathbb{R}^N\right)$. An
  $\L^1$-solution of \eqref{eq1} is a function $u$ in  
  $\C\left ([0,\infty); \L^1\left(\mathbb{R}^N\right)\right)$ such that
    \eqref{eq1} holds in the sense of distributions, or equivalently, if
    for every $t>0$, $u(t) \in \L^1\left(\mathbb{R}^N\right)$ and
    \begin{equation}\label{solution} u(t)=\displaystyle f
      +\displaystyle\int_0^t(J\ast \Gamma(u)(s)-\Gamma(u)(s))\ds,\quad
      \mbox{a.e.} \end{equation}	 \end{definition}

\begin{remark} If  $u$ is an $\L^1$-solution, then $u\in
  \L^1([0,T]; \L^1(\mathbb{R^N}))$.
 for all $T>0$. Hence,
  \eqref{eq1} holds, not only in the sense of distributions, but also
  a.e., and $u$ is said to be a strong solution. Moreover, since
  $\Gamma(u)\in \C\left([0,\infty);
    \L^1\left(\mathbb{R}^N\right)\right)$, we also have $u\in
    \C^1\left([0,\infty); \L^1\left(\mathbb{R}^N\right)\right)$, and the
      equation holds a.e. in $x$ for all $t \ge0$.  \end{remark}

\begin{teo}\label{existence_solution_fix_point} For any $f\in
  \L^1\left(\mathbb{R}^N\right)$, there exists a unique $\L^1$-solution
  of \eqref{eq1}.  \end{teo} 
  
  \begin{proof} Let $\mathcal{B}_{t_0}$ be
    the Banach space consisting of the functions $u\in
    \C\left([0,t_0];\L^1\left(\mathbb{R}^N\right)\right)$ endowed with
    the norm, \[ \||u\||=\max\limits_{0\le t\le
    t_0}\|u(t)\|_{\L^1\left(\mathbb{R}^N\right)}.  \] 
    For any given $f\in\L^1(\R^{N})$, we define the
    operator $\mathcal{T}_{f}:\mathcal{B}_{t_0} \rightarrow
    \mathcal{B}_{t_0}$ through \[
      \left(\mathcal{T}_fu\right)(t)=f+\displaystyle\displaystyle\int_0^t
      \left(J\ast \Gamma(u)(s)-\Gamma (u)(s)\right)\ds.  \] Since
      $\Gamma(u)$ is Lipschitz continuous, we have the estimate
      \[\begin{aligned}
      \||\mathcal{T}_fu-\mathcal{T}_fv\||& \leq 
      \int_{0}^{t_{0}}\int_{\R^{N}}
      \bigg(J\ast\big|\Gamma(u)-\Gamma(v)\big| + 
      |\Gamma(u)-\Gamma(v)|\bigg)\dx\ds \\ & \leq 
      2\int_{0}^{t_{0}}\int_{\R^{N}}\big|u-v\big|\dx\ds{}
      \leq 2t_{0}\||u-v\||\,.{}
      \end{aligned}
      \]
      Hence if $t_0<1/2$, the operator $\mathcal{T}_{f}$ turns out to be contractive.
      
      Existence and uniqueness
      in the time interval $[0,t_0]$ follow by using Banach's fixed
      point Theorem. The length of the existence and uniqueness time
      interval does not depend on the initial data, so, we can iterate
      the argument to extend the result to all positive times by a 
      standard procedure, and we end up with a solution in $\C\left([0,\infty);\L^1\left(\mathbb{R}^N\right)\right)$.
    \end{proof} 

\noindent{\bf Conservation of energy of the $\L^1$-solutions.}

\begin{teo} Let $f\in \L^1\left(\mathbb{R}^N\right)$. The
  $\L^1$-solution $u$ to \eqref{eq1} satisfies 
  \[
    \displaystyle\int_{\mathbb{R}^N}u(t)=\displaystyle\int_{\mathbb{R}^N}f,\quad
    \mbox{for every } t>0.  
  \] 
\end{teo}

\begin{proof} Since $u(t)\in
      \L^1\left(\mathbb{R}^N\right)$ for any $t\geq0$, we integrate
      equation (\ref{solution}) in space:
      \[{}
      \int_{\R^{N}}u(t)=\int_{\R^{N}}f + \int_{0}^{t}
      \bigg(\int_{\R^{N}} J\ast u - \int_{\R^{N}}u\bigg)\ds\,.{}
      \]
      By Fubini's Theorem, $\int J\ast 
      u=\int J\cdot\int u=\int u$ (where the integrals are taken over 
      all $\R^{N}$), which yields the result.
\end{proof}

\noindent {\bf $\L^1$-contraction property for $\L^1$-solutions.}

In order to obtain it, we need first to approximate the graph
$\Gamma(s)$ by a sequence of strictly monotone $\Gamma_n(s)$ such that:

\begin{enumerate} 
     \item[$(i)$] there
      is a constant $L$ independent of $n$ such that
      $|\Gamma_n(s)-\Gamma_n(t)|\le L|s-t|$, for all $n\in\mathbb{N}$;
	
    \item[$(ii)$] for all $n\in\mathbb{N}$, $\Gamma_n(0)=0$ and $\Gamma_n$
	  is strictly increasing on $(-\infty, \infty)$;
	
    \item[$(iii)$] $|\Gamma_n(s)|\le s$, for all $n\in\mathbb{N}$ and $s\ge
	  0$;
	
   \item[$(iv)$] $\Gamma_n\rightarrow \Gamma$ as $n\rightarrow\infty$
	  uniformly in $(-\infty,\infty)$.  
\end{enumerate}

Take for instance  \[ \Gamma_n(s)= \left\{ \begin{array}{ll} (s+1),
    & \mbox{for } s<\frac{-n-1}{n} \smallskip\\
    \displaystyle\frac{s}{n+1},		& \mbox{for } \frac{-n-1}{n}\le
    s \le \frac{n+1}{n} \smallskip\\ (s-1), & \mbox{for } s>
    \frac{n+1}{n}.  \end{array} \right.   \]

Since $\Gamma_n$ is Lipschitz, for any $f\in
\L^1\left(\mathbb{R}^N\right)$ and any $n\in\mathbb{N}$ there exists a
unique $\L^1$-solution $u_n\in \C\left([0,\infty);
  \L^1\left(\mathbb{R}^N\right)\right)$ of the approximate problem
  \begin{equation}\label{aproximate_eq1} \partial_t u_n=J\ast
    \Gamma_n(u_n)-\Gamma_n(u_n) \end{equation} with initial data $u_n(0) =
    f$. The proof is just like the one of Theorem
    \ref{existence_solution_fix_point}. Moreover, $\Gamma(u_n)\in
    \C\left([0,\infty); \L^1\left(\mathbb{R}^N\right)\right)$, and,
      hence, $u_n\in \C^1([0,\infty); \L^1\left(\mathbb{R}^N\right))$.
	Conservation of energy also holds, the calculations are the same 
	as for $\L^{1}$-solutions above.\\

Now we state the $\L^{1}$-contraction property for the approximate 
problem:

\begin{lema} Let $u_{n,1}$ and $u_{n,2}$ be two $\L^1$-solutions of
  \eqref{aproximate_eq1} with initial data   $f_1,\, f_2 \in
  \L^1\left(\mathbb{R}^N\right)$. Then,
  \begin{equation}\label{aproximation_L1_contraction}
    \|(u_{n,1}-u_{n,2})(t)\|_{\L^1\left(\mathbb{R}^N\right)}\le
    \|f_1-f_2\|_{\L^1\left(\mathbb{R}^N\right)}, \quad \forall t\ge 0.
  \end{equation} \end{lema}

\begin{proof}
	The proof is done in \cite[Lem 2.4]{BCQ2012}: we begin by proving 
	a contraction property for the positive part 
	$(u_{n,1}-u_{n,2})_{+}$. To do so, we subtract the equations 
	for $u_{n,1}$ and $u_{n,2}$ and multiply by $\mathds{1}_{\{u_{n,1}>u_{n,2}\}}$. 
	Since $u_{n,1}-u_{n,2}\in\C^1([0,\infty);\L^1(\mathbb{R}^N))$, then 
    $$
    \partial_t (u_{n,1}-u_{n,2})\mathds{1}_{\{u_{n,1}>u_{n,2}\}}=\partial_t (u_{n,1}-u_{n,2})_+.
    $$
    On the other hand, since $0\leq\mathds{1}_{\{u_{n,1}>u_{n,2}\}}\leq 1$, we have
    $$
    J\ast(\Gamma_n(u_{n,1})-\Gamma_n(u_{n,2}))\mathds{1}_{\{u_{n,1}>u_{n,2}\}}
    \leq J\ast(\Gamma_n(u_{n,1})-\Gamma_n(u_{n,2}))_+.
    $$
    Finally, since $\Gamma_n$ is strictly monotone,
    $\mathds{1}_{\{u_{n,1}>u_{n,2}\}}=\mathds{1}_{\{\Gamma_n(u_{n,1})>\Gamma_n(u_{n,2})\}}$.
    Thus,
    $$
    (\Gamma_n(u_{n,1})-\Gamma_n(u_{n,2}))\mathds{1}_{\{u_{n,1}>u_{n,2}\}}=(\Gamma_n(u_{n,1})-\Gamma_n(u_{n,2}))_+
    .
    $$
    We end up with
    $$
    \partial_t (u_{n,1}-u_{n,2})_+\leq J\ast(\Gamma_n(u_{n,1})-\Gamma_n(u_{n,2}))_+-
    (\Gamma_n(u_{n,1})-\Gamma_n(u_{n,2}))_+.
    $$
    Integrating in space, and using Fubini's Theorem, which can be
    applied, since
    $(\Gamma_n(u_{n,1}(t))-\Gamma_n(u_{n,2}(t)))_+\in\L^1(\R^N)$, we get
    $$
    \partial_t \int_{\mathbb{R}^N}(u_{n,1}-u_{n,2})_+(t)\le0\,,
    $$
    which implies $$\int_{\R^{N}}(u_{n,1}-u_{n,2})_{+}\dx\leq
    \int_{\R^{N}}(f_{1}-f_{2})_{+}\dx\,.$$
    Then, a similar computation gives the contraction for the negative 
    parts, so that the $\L^{1}$-contraction holds.
\end{proof}

Then we deduce the $\L^1$-contraction property for the original 
problem after passing to the limit:

\begin{corolario}\label{contract_prop_L_1} Let $u_1$ and $u_2$ be two $\L^1$-solutions of
  \eqref{eq1} with initial data $f_1,\, f_2\in \L^1
  \left(\mathbb{R}^N\right)$. Then for every $t\ge 0$,
  \begin{equation}\label{L1_contraction}
    \|(u_{1}-u_{2})(t)\|_{\L^1\left(\mathbb{R}^N\right)}\le
    \|f_1-f_2\|_{\L^1\left(\mathbb{R}^N\right)}\,, 
    \end{equation}
    and the same result holds for the positive/negative parts of 
    $(u_{1}-u_{2})$.
  \end{corolario}

\begin{proof}
	Passing to the limit in the approximated problems requires some 
	compactness argument which is obtained through the
	Fréchet-Kolmogorov criterium. The details are 
	in \cite[Cor. 2.5]{BCQ2012}, and do not depend on the specific 
	form of the function $\Gamma(\cdot)$ so we skip the proof.
\end{proof}

The following Lemma shows that the positive and negative parts of
$\Gamma(u)$ are subcaloric:
\begin{lema}\label{gamma_subcaloric} Let $f \in
	  \L^1\left(\mathbb{R}^N\right)$ and $u$ the corresponding
	  $\L^1$-solution. Then the functions $\big(\Gamma(u)\big)_-\,,\
	  \big(\Gamma(u)\big)_+$ and $|\Gamma(u)|$ all satisfy the inequality:
	  \begin{equation*} \chi_t\le J\ast \chi-\chi\quad \text{a.e. in
	    }\mathbb{R}^N \times (0,\infty)\,.  \end{equation*} 
 \end{lema}
 \begin{proof} We do the computation for $\chi=|\Gamma(u)|$, with the
      proof being the same for the other functions. Since $u \in
      \C^1([0,\infty); \L^1\left(\mathbb{R}^N\right))$, we have, \[
	\begin{array}{ll} \left|\Gamma(u)\right|_t  	&	=
	  \left(\left( |u|-1\right)_+\right)_t\\ &
	  =\,\sign(u)\,u_t\\ &
	  =\sign(u)\,J\ast\Gamma(u)-\sign(u)\Gamma(u)\quad\mbox{a.e.}
	\end{array} \] On the set $\{|u|\le 1\}$ we have
	$\left|\Gamma(u)\right|=\left|\Gamma(u)\right|_t\!=\!0$ while $0\le J \ast
	\left|\Gamma(u)\right| $, so that the following inequality 
	necessarily holds:
	$$\left|\Gamma(u)\right|_t \leq
	    J\ast \left|\Gamma(u)\right|-\left|\Gamma(u)\right|.$$
	On the set $\{|u|>1\}$, using that $|\sign(u)|=1$ we get also
	\[
	  \begin{array}{ll} \left|\Gamma(u)\right|_t 	&
	    =\sign(u)J\ast\Gamma(u)-\sign(u)\Gamma(u)\\ & 	\le
	    J\ast \left|\Gamma(u)\right|-\left|\Gamma(u)\right|.
	  \end{array} \] 
	 Hence in any case, we obtain the result.
\end{proof}

This property allows to estimate the size of the solution in terms of
the $\L^{\infty}$-norm of the initial data. 
\begin{lema} \label{lem.comp.linfty}
	Let $f \in \L^1\left(\mathbb{R}^N\right) \cap
  \L^{\infty}\left(\mathbb{R}^N\right)$. Then the $\L^1$- solution $u$
  of \eqref{eq1} satisfies $\|u(t)\|_{\L^{\infty}(\mathbb{R}^N)} \le
  \|f\|_{\L^{\infty}(\mathbb{R} ^N)}$ for any $t > 0$. Moreover, \[\lim
    \sup _{t\to\infty} u(t) \le 1 \;\mbox{ and } \;\lim \inf
    _{t\to\infty} u(t) \ge -1 \mbox{ a.e. in } \mathbb{R}^N.\]
\end{lema}
\begin{proof}
	The proof follows the same arguments as in \cite[Lem 2.7]{BCQ2012}:
    first, the result is obvious if $\|f\|_{\L^\infty(\R^N)}\leq 1$, 
    since in this case $u(t)=f$ for any $t>0$. So let us assume that 
    $\|f\|_{\L^\infty(\R^N)}>1$. Since $\chi=|\Gamma(u)|$ is subcaloric 
    (by Lemma \ref{gamma_subcaloric}), 
    we may compare it with the solution $V$ of the following problem:
    $$
    V_t=J\ast V -V,\quad V(0)=|\Gamma(f)|\in\L^1(\mathbb{R}^N)\cap\L^\infty(\mathbb{R}^N).
    $$
    We first use the comparison principle in $\L^{\infty}$ (see
    \cite[Prop. 3.1]{BrandleChasseigneFerreira2011}) with constants 
    (which are solutions):
    $0\leq V(t)\leq\|V(0)\|_{\infty}=\|\Gamma(f)\|_{\infty}$. 
    Now, using again the comparison principle for bounded sub/super 
    solutions, we obtain  
    $$0\leq 
    \|\chi(t)\|_{\L^\infty(\R^N)}\leq 
    \|V(t)\|_{\L^\infty(\R^N)}\leq \|\Gamma(f)\|_{\L^\infty(\R^N)}= 
    \|f\|_{\L^\infty(\R^N)}-1.$$ 
    Therefore,  $\|u(t)\|_{\L^\infty(\R^N)}\leq 
    1+\|\chi(t)\|_{\L^\infty(\R^N)}\leq\|f\|_{\L^\infty(\R^N)}$. 
    Moreover, using the results from~\cite{IgnatRossi}, we obtain that 
    $V$, and hence the solution $v$, goes to zero asymptotically like 
    $ct^{-N/2}$, so that $\Gamma(u)\to0$ almost everywhere, which implies 
    the result.
\end{proof}


\subsection{BC theory.}

We now develop a theory in the class  
$\mathrm{BC}\left(\mathbb{R}^N\right)$ of continuous and bounded functions
whenever the initial data $f$ belongs to that class.

\begin{definition} Let $f \in \mathrm{BC}\left(\mathbb{R}^N\right)$. The function
  $u$ is a $\mathrm{BC}$-solution of \eqref{eq1} if $u\in
  \mathrm{BC}\left(\mathbb {R}^N \times [0, T ]\right)$ for all $T\in (0,\infty)$
  and  \[ u(x, t) = f(x)+\displaystyle\displaystyle\int_0^t \left(J \ast
    \Gamma(u)(x, s) - \Gamma(u)(x,s)\right)\ds,  \] for all $x\in
    \mathbb{R}^N$ and $t\in [0,\infty)$.  
 \end{definition}

In particular, a BC-solution $u$ is continuous in $[0,\infty)\times\R^{N}$ 
and $u_t$ is also continuous in $(0,\infty)\times\R^{N}$.  Hence
equation \eqref{eq1} is satisfied for all $x$ and $t$, and $u$ is a
classical solution. 

\begin{teo}\label{fix_point_bc} For any $f \in
  \mathrm{BC}\left(\mathbb{R}^N\right)$ there exists a unique 
  $\mathrm{BC}$-solution of \eqref{eq1}.  
\end{teo}
\begin{proof}
 The proof is obtained through a fixed-point argument exactly as for 
 $\L^{1}$-solutions, except that we consider the operator 
 $\mathcal{T}_{f}$ as acting from $\mathrm{BC}([0,t_{0}]\times\R^{N})$ 
 into $\mathrm{BC}([0,t_{0}]\times\R^{N})$. The estimates are done 
 using the sup norm in space and time instead of the sup of the 
 $L^{1}$-norm but the result is the same: if $t_{0}$ is small enough, 
 then we have a contractive operator which allows to construct a 
 unique solution on $[0,t_{0}]$. The we iterate the process to get a 
 bounded and continuous solution on $[0,T]\times\R^{N}$ for any $T>0$.
 \end{proof}

\noindent Notice that BC-solutions depend continuously on the initial 
data, on any finite time interval:
\begin{lema}\label{cota_inf_BC} Let $u_1$ and $u_2$ be the BC-solutions
  with initial data respectively $f_1,\, f_2 \in BC\left(\mathbb
  {R}^N\right)$. Then, for all $T \in (0,\infty)$ there exists a
  constant $C = C(T )$ such that \[ \max\limits_{x\in\mathbb{R}^N}|u_1 -
    u_2|(x, t) \le C(T)\max\limits_{x\in\mathbb{R}^N} |f_1 - f_2|
    (x),\quad t \in [0, T ].  \] \end{lema}
    
\begin{proof}
	See \cite[Lem 2.10]{BCQ2012}.
\end{proof} 


\subsection{Free boundaries} In the sequel, unless we say explicitly
something different, we will be dealing with $\L^1$-solutions. Since the
functions we are handling are in general not continuous in the space
variable, their support has to be considered in the distributional
sense. To be precise, for any locally integrable and nonnegative
function $g$ in $\mathbb{R}^N$, we can consider the distribution $T_g$
associated to the function $g$. Then the distributional support of $g$,
$\suppD(g)$ is defined as the support of $T_g$:

\noindent$\suppD(g)\, := \mathbb{R}^N\! \setminus\! \mathcal{O}\mbox{,
where } \mathcal{O} \subset  \mathbb{R}^N$ is the biggest open set such
that $\left.T_g\right|\mathcal{O}=0$.

\noindent In the case of nonnegative functions $g$, this means that $x
\in \suppD(g)$ if and only if  \[ \forall \varphi\in  \C^{\infty}_c
  (\mathbb{R}^N),\; \varphi \ge 0 \mbox{ and } \varphi(x)>0, \; \mbox{
  happens that
}\,\displaystyle\displaystyle\int_{\mathbb{R}^N}g(y)\varphi(y) dy > 0.
\] If $g$ is continuous, then the support of $g$ is nothing but the
usual closure of the positivity set, $\suppD(g) =\overline{ \{g >
0\}}$.\\

We first prove that the solution does not move far away from the support
of $\Gamma(u)$.  \begin{lema}\label{soporte_u_t} Let $f \in
  \L^1\left(\mathbb{R}^N\right)$. Then, $\suppD(u_t(t)) \subset
  \suppD(\Gamma(u)(t)) + B_{R_J}$ for any $t > 0$.  \end{lema}
  \begin{proof} Recall first that the equation holds down to $t = 0$ so
    that we may  consider here $t \ge 0$ (and not only $t >
    0$). Let $\varphi\in \C^{\infty}_c (A^c)$, where  $A =
    \suppD(\Gamma(u)(t)) + B_ {R_J} $. Notice that the support of $J
    \ast \Gamma(u)$ (which is a continuous function) lies inside $A$, so
    that \[ \displaystyle\displaystyle\int_{\mathbb{R}^N}(J \ast
      \Gamma(u))\varphi = 0.  \] Similarly, the supports of $\Gamma(u)$
      and $\varphi$ do not intersect, so that \[
	\displaystyle\displaystyle\int_{\mathbb{R}^N}u_t\varphi
	=\displaystyle\displaystyle\int_{\mathbb{R} ^N}(J \ast
	\Gamma(u))\varphi
	-\displaystyle\displaystyle\int_{\mathbb{R}^N}\Gamma(u)\varphi =
	0, \] which means that the support of $u_t$ is contained in $A$.
      \end{proof}

The following Theorem gives a control of the support 
of the solution $u(t)$ and the corresponding 
temperature $\Gamma(u)(t)$.

\begin{teo} \label{teo:compact}
	Let $f\in \L^1(\mathbb{R}^N)$ be compactly supported. Then,
  for any $t>0$, the solution $u(t)$ and the corresponding temperature
  $\Gamma(u)(t)$ are compactly supported.  
  \end{teo} 
  \begin{proof} 
  
\noindent{\sc Estimate of the support of $\Gamma(u)$.} Since
$|\Gamma(u)|$ is subcaloric, we have that $\|\Gamma(u)\|_{\L^1
(\Omega)}\le \|\Gamma(f)\|_{\L^1(\Omega)} $, then \[ (J \ast
  \Gamma(u))(x,t)\le
  \|J\|_{L^{\infty}(\mathbb{R}^N)}\|\Gamma(u)\|_{\L^1(\mathbb{R}^N)} \le
  \|J\|_{L^{\infty}(\mathbb{R}^N)} \|\Gamma(f)\|_{\L^1(\mathbb{R}^N)}.
  \] We denote $c_0=
  \|J\|_{L^{\infty}(\mathbb{R}^N)}\|\Gamma(f)\|_{\L^1(\mathbb{R}^N)}$.
  Multiplying (\ref{solution}) by a nonnegative test function $\varphi
  \in \C^{\infty}_c ((\suppD f)^c)$ and integrating in space and time we
  have \[ \displaystyle\displaystyle\int_{\mathbb{R}^N}|u(t)|\varphi \le
    \displaystyle\int^ t_0\displaystyle\int_{\mathbb{R}^N}(J \ast \Gamma
    (u))\varphi \le c_0 t\displaystyle\int_{\mathbb{R}^N}\varphi.  \]
    Taking $t_0 = 1/c_0$, we get
    $\displaystyle\int_{\mathbb{R}^N}(\left|u(t)\right| - 1)\varphi \le
    0$ for all $t \in [0, t_0]$. Using an approximation
    $\varphi\chi_n$ where $\chi_n \to \sign_+(|u|-1)$, we deduce that
    $\displaystyle\int_{\mathbb{R}^N}|\Gamma(u)| \varphi \le 0$, so that
    \begin{equation}\label{soporte_gamma_u} \suppD(\Gamma(u)(t)) \subset
      \suppD(f),\quad \mbox{for all } t \in [0, t_0].  \end{equation}

\noindent{\sc Estimate of the support of $u$.} Thanks to Lemma
\ref{soporte_u_t} we know that $\suppD(u_t(t)) \subset
\suppD(\Gamma(u)(t)) + B_{R_J}\subset \suppD(f) + B_{R_J} ,\quad
\mbox{for all } t \in [0, t_0]$.  This means that for any $\varphi\in
\C^{\infty}_c ((\suppD(f) + B_{R_J} )^c)$, we have, \[
  \displaystyle\displaystyle\int_{\mathbb{R}^N}u\varphi={}
  \displaystyle\int_0^t\displaystyle\int_{\mathbb{R}^N}u_t\varphi=0,\quad
  \mbox{for all } t \in [0, t_0] \] that is,
  \begin{equation}\label{soporte_u} \suppD(u(t)) \subset \suppD(f)+
    B_{R_J} ,\quad \mbox{for all } t \in [0, t_0].  \end{equation}
    \noindent {\sc Iteration.} Consider now the initial data
    $u_0=u(t_0)$, whose support satisfies that, \[ \suppD(u(t_0))
      \subset \suppD(f)+ B_{R_J}, \] then, thanks to
      \eqref{soporte_gamma_u} and \eqref{soporte_u}, \[ \suppD(u(t))
	\subset \suppD(f)+ 2\,B_{R_J},\quad\mbox{ for all } t \in [0,
	  2\,t_0].  \] Iterating this process we arrive to, \[
	    \suppD(\Gamma(u)(t)) \subset \suppD(f)+
	    nB_{R_J},\quad\mbox{with } n= \lfloor t/t_0\rfloor, \] and
	    \[ \suppD(u(t)) \subset \suppD(f)+ nB_{R_J},\quad\mbox{with
	    } n= \lfloor t/t_0\rfloor+1, \] where $\lfloor x\rfloor$ is
	    the integer part of $x$.  \end{proof}

The last results have counterparts for $\mathrm{BC}$-solutions:

\begin{teo}\label{th:fsp.bc}
Let $f\in\mathrm{BC}(\R^N)$, and let $u$ be the corresponding
$\mathrm{BC}$-solution. Then, noting $v=\Gamma(u)$ we have:
\begin{itemize}
\item[\rm(i)]
    $u_t(x,t)=0$ for any $x\notin (\supp(v(\cdot,t))+B_{R_J})$, $t\ge0$.

\item[\rm(ii)] If $\sup_{|x|\ge R}|f(x)|<1$ for some $R>0$,   then 
$v(\cdot,t)$ is compactly supported for all $t>0$. If moreover 
$f\in\C_{\rm c}(\R^N)$, then $u(\cdot,t)$ is also compactly supported 
for all $t>0$. 
    \end{itemize}
\end{teo}
\begin{proof}
(i) The proof is  similar (though even easier, since the supports
are understood in the classical sense) to the one for
$\L^1$-solutions.

\noindent (ii) Since $\chi=|\Gamma(u)|$ is subcaloric,  we get
    $$
    \Big|\big(J*\Gamma(u)\big)(x,t)\Big|\leq \|J\|_{\L^1(\R^N)}\|\Gamma(u)(t)\|_{\L^\infty(\R^N)}\leq
    \|\Gamma(f)\|_{\L^\infty(\R^N)}.
    $$
   This estimate comes from comparison in $\L^{\infty}$ with 
   constants, exactly as in Lemma~\ref{lem.comp.linfty}
Therefore, from the integral equation, \eqref{solution} for $|x|\ge 
R$ we have 
\begin{equation}
  \label{eq:bound.for.u}
  \left\{
  \begin{aligned}
u(x,t) &\le f(x)+t\|\Gamma(f)\|_{\L^\infty(\R^N)}\le\sup_{|x|\ge R} 
|f(x)| +t\|\Gamma(f)\|_{\L^\infty(\R^N)}\,, \\
u(x,t) &\ge f(x)- t\|\Gamma(f)\|_{\L^\infty(\R^N)}\ge-\sup_{|x|\ge R} 
|f(x)| -t\|\Gamma(f)\|_{\L^\infty(\R^N)}\,.
\end{aligned}\right.
\end{equation}
Thus, for all $|x|\ge R$ and $t\le(1-\sup_{|x|\ge
R}|f(x)|)/(2\|\Gamma(f)\|_{\L^\infty(\R^N)})$ we have $-1<u(x,t)<1$. 
Hence, for such $x,t$, we have $v(x,t)=0$.  
Then, by (i), $u(x,t)=f(x)$ for all $|x|\ge R+R_J$
and $t=(1-\sup_{|x|\ge R}|f(x)|)/(2\|\Gamma(f)\|_{\L^\infty(\R^N)})$. 
We finally proceed by iteration to get the result for all times.
\end{proof}


\subsection {$\L^1$-solutions that are continuous.}

As a corollary of the control of the supports, we will prove that if the
initial data is in $\L^1\left(\mathbb{R}^N\right) \cap
\C_0(\mathbb{R}^N)$, with $\C_0(\mathbb{R}^N)=\{\varphi\in
  \C(\mathbb{R}^N):\; \varphi\to 0\;\mbox{as } |x|\to\infty\}$, then the
  $\L^1$-solution is in fact continuous. We start by considering the
  case where $f$ is continuous and compactly supported, i.e. in $ \C_c(\mathbb{R}^N)$.

\begin{lema} \label{continuous_sol} Let $f \in
  \L^1\left(\mathbb{R}^N\right) \cap \C_c(\mathbb{R}^N)$. 
  Then the corresponding $\L^1$-solution is
  continuous in $[0,\infty)\times\R^{N}$.  \end{lema}
 \begin{proof}
Since a BC-solution with a continuous and compactly supported initial 
data remains compactly supported in space for all times (see 
Theorem~\ref{th:fsp.bc}), it is also 
integrable in space for all times. Moreover, $u\in 
\C([0,T];L^1(\mathbb{R}^N))$. Hence, by uniqueness it coincides with the 
$\L^1$-solution with the same initial data. In other terms, the 
$\L^{1}$-solution is continuous.
\end{proof}

We now turn to the general case.
 
 \begin{prop} Let $f \in \L^1\left(\mathbb{R}^N\right) \cap
   \C_0(\mathbb{R}^N)$.  Then the corresponding $\L^1$-solution is
   continuous in $[0,\infty)\times\R^{N}$.  \end{prop}
\begin{proof}
  Let $f_n$ be a sequence of continuous and compactly 
  supported functions such that
\[ \|f_n - f\|_{L^{\infty}(\mathbb{R}^N)} <\frac{1}{n},\qquad \|f_n - f\|_{L^1
(\mathbb{R}^N)} <\frac{1}{n}.\]
Let $u^1_n,\, u^1$ be the $\L^1$-solutions with initial data respectively $f_n$ 
and $f$, and $u^c_n,\; u^c$ the corresponding BC-solutions. We know by 
Lemma \ref{continuous_sol} that  $u^1_n = u^c_n$. Then, using the
$\L^1$-contraction property for $\L^1$-solutions, we have that 
$$\|u^1_n - u^1\|_{L^1(\mathbb{R}^N\times [0,T])}\to 0$$ 
for any $T \in [0,\infty)$. Moreover, by 
Lemma \ref{cota_inf_BC}, $\|u^1_n - u^c\|_{L^{\infty}([0,T], 
L^{\infty}(\mathbb{R}^N))} \to 0$. Hence we have in the limit 
$u^{1}=u^{c}$ which proves the result.
\end{proof}


\section{First results concerning the asymptotic behaviour}
\label{sect:not.interact}

In the following three sections we study the asymptotic behaviour of the
solutions of the two-phase Stefan problem, with different sign-changing
initial data chosen in such a way that the solutions, 
$u(t)$, satisfy either: 
\begin{enumerate} \renewcommand{\labelenumi}{$(\roman{enumi})$} 
  \item the positive and negative part not interact, in any time $t>0$; 
  \item the positive and negative temperature $v=\Gamma(u)$ do not
  interact, in any time $t>0$; 
  \item the positive and negative part of $u$ interact but the solution
    is driven by the one-phase Stefan regime after some time. 
\end{enumerate}

In order to describe the asymptotic behaviour, we write the initial 
data as $$f=f_+-f_-,$$ separating the positive and negative parts 
where we recall the notations $f_{+}=\max(f,0)$ and $f_{-}=\max(-f,0)$.{}

Let us first introduce
the following solutions: the solution $\mathbb{U}^+$, corresponding to
the initial data $\mathbb{U}^+(0)=f_+$ and the solution $\mathbb{U}^-$,
corresponding to the initial data $\mathbb{U}^-(0)=f_-$\,.

\begin{lema} The functions $\mathbb{U}^+$ and $\mathbb{U}^-$ are
  solutions of the one-phase Stefan problem: $$\partial_t u=J\ast(u-1)_+
  - (u-1)_+\,.$$ \end{lema}
  
 \begin{proof} By comparison in $\L^1$ for
    the two-phase Stefan problem, we know that $\mathbb{U}^+$ and
    $\mathbb{U}^-$ are nonnegative because their respective initial data
    are nonnegative.  Hence, for any $(x,t)$ we have in fact
    $\Gamma(\mathbb{U}^+(x,t))=(\mathbb{U}(x,t)-1)_+$.  Thus, the
    equation for $\mathbb{U}^+$ reduces to the one-phase Stefan problem.
    The same happens for $\mathbb{U}^-$.  \end{proof}

\begin{remark}
  Since $\mathbb{U}^+$
  is a solution of the one-phase Stefan problem, the supports of
  $\;\mathbb{U}^+$ and $\Gamma\left(\mathbb{U}^+\right)$ are
  nondecreasing 
  \begin{equation}\label{retention_property_one_phase}
  \begin{array}{ll}
    \supp_{\mathcal{D}'}(\mathbb{U}^+(s))\subset
    \supp_{\mathcal{D}'}(\mathbb{U}^+(t)), & 0\le
    s\le t\\
    \supp_{\mathcal{D}'}(\Gamma\left(\mathbb{U}^+\right)(s))\subset
    \supp_{\mathcal{D}'}(\Gamma\left(\mathbb{U}^+\right)(t)), & 0\le
    s\le t.  
    \end{array}
    \end{equation} We denote this property as {\bf retention}.
    It is satisfied also for $\mathbb{U}^-$ and
    $\Gamma\left(\mathbb{U}^-\right)$.  \end{remark}

Using the results concerning the asymptotic behaviour studied in
\cite{BCQ2012}, we know that in particular if $f$ satisfies the
hypothesis of \cite[Lem. 3.9.]{BCQ2012}, $\mathbb{U}^+$ and
$\mathbb{U}^-$ have limits as $t\to\infty$ which are obtained by means
of the projection operator $\proj{}$. 
We recall that this operator maps any nonnegative
initial data $f$ to $\proj{f}$, which is the unique solution to a
non-local obstacle problem at level one (see \cite[p. 23]{BCQ2012}). For
$\mathbb{U}^+$, the limit is $\proj{f_+}$ and for $\mathbb{U}^-$, the
limit is $\proj{f_-}$.  Now the link with our problem is the following:

\begin{lema}\label{lem:comp.upm} For any $t>0$,
  $-\mathbb{U}^-(t)\leq-u_-(t)\leq u(t) \leq u_+(t)\leq
  \mathbb{U}^+(t)\,.$ \end{lema} \begin{proof} This result follows from
    a simple comparison result in $\L^1$: since initially we have
    $\mathbb{U}^+(0)=f_+\geq u(0)$, it is clear that for any $t>0$,
    $\mathbb{U}^+(t)\geq u(t)$.  On the other hand, since
    $\mathbb{U}^+(0)=f_+\geq0$, we have also for any $t>0$,
    $\mathbb{U}^+(t)\geq 0$. Hence for any $t>0$, $\mathbb{U}^+(t)\geq
    u_+(t)$.
	
	    The other inequalities are obtained the same way.
	  \end{proof}

	This comparison allows us to prove that the asymptotic limit is
	well-defined: \begin{prop} \label{prop_hip_3_9} Let us assume
	  that $f\in\L^1(\mathbb{R^N})$ if $N\ge 3$, for low dimensions,
	  if $N=1$ or $N=2$, $J$ is non increasing in the radial
	  variable, and $f_+\le g_1$, $f_-\le g_2$ for some $g_1,\,
	  g_2\in\L^1\left(\mathbb{R}^N\right)\cap \C_0(\mathbb{R}^N)$,
	  radial and strictly decreasing in the radial variable.  Then
	  the following limit is defined in $\L^1(\R^N)$:
	  $$u_\infty(x):=\lim_{t\to\infty}u(x,t)\,.$$ \end{prop}
	  \begin{proof} Integrating the equation \eqref{eq1} in time we
	    get $$u(t)=f+\int_0^t
	    J\ast\Gamma(u)(s)\ds-\int_0^t\Gamma(u)(s)\ds\,.$$ Then we
	    recall that under the hypotheses of this proposition, the
	    integrals $$\int_0^t(\mathbb{U}^+(s)-1)_+\ds\quad \text{and}
	    \int_0^t(\mathbb{U}^-(s)-1)_+\ds$$ converge in $\L^1$ as
	    $t\to\infty$ (see \cite[Cor. 3.10, 3.11]{BCQ2012}).  Using
	    the estimate
	    \[|\Gamma(u)|\leq\max\big((\mathbb{U}^+-1)_+\,;\,(\mathbb{U}^--1)_+\big),\]
	    we deduce that the right-hand side of the integrated
	    equation has a limit as $t\to\infty$. Hence we deduce that
	    $u(t)$ has a limit in $\L^1(\R^N)$ which can be written as:
	    $$\lim_{t\to\infty}u(t)=f+\int_0^\infty
	    J\ast\Gamma(u)(s)\ds-\int_0^\infty\Gamma(u)(s)\ds:=u_\infty(x)\,.$$
	  \end{proof}

The question is now to identify this limit $u_\infty$ and we begin with
a simple case when the positive and negative parts never interact:
\begin{lema} Let us assume that $J$ and $f$ satisfy the hypothesis of
  Proposition \ref{prop_hip_3_9}, and that
  $$\dist\big(\supp(\proj{f_+}),\supp(\proj{f_-})\big)\geq r>0\,.$$ Then
  for any $t>0$,  $\dist\big(\supp(u_-(t)),\supp(u_+(t))\big)\geq r\,.$
\end{lema} \begin{proof} By the retention property 
  \eqref{retention_property_one_phase} for $\mathbb{U}^+$ and
  $\mathbb{U}^-$, we first know that for any $t>0$,
  $\dist\big(\supp(\mathbb{U}^+(t)),\supp(\mathbb{U}^-(t))\big)\geq
  r\,.$ Then, since $0\leq u_+(t)\leq\mathbb{U}^+(t)$, the support of
  $u_+(t)$ is contained inside the one of $\mathbb{U}^+(t)$.  The same
  holds for $u_-(t)$ and $\mathbb{U}^-(t)$ so that finally, the supports
  of $u_-(t)$ and $u_+(t)$ are necessarily at distance at least $r$.
\end{proof}

\begin{teo} Let us assume that $J$ and $f$ satisfy the hypothesis of
  Proposition \ref{prop_hip_3_9} and that
  $$\dist\big(\supp(\proj{f_+}),\supp(\proj{f_-})\big)>2R_J\,.$$ Then
  the solution with initial data $f$ is given by
  $u(t)=\mathbb{U}^+(t)-\mathbb{U}^-(t)$, and the asymptotic behaviour
  is given by $$u_\infty(x)=\proj{f_+}(x)-\proj{f_-}(x)\,.$$ \end{teo}
  \begin{proof} Let us define $\mathbb{U}:=\mathbb{U}^+-\mathbb{U}^-$.
    Since the supports of $\mathbb{U}^+(t)$ and $\mathbb{U}^-(t)$ are
    always at distance greater that $2R_J$, we can write
    $\mathbb{U}(t)=\mathbb{U}^+(t)-\mathbb{U}^-(t)$.  Moreover, the
    convolution $J\ast\Gamma(\mathbb{U}(t))$ is either equal to $J\ast
    \Gamma(\mathbb{U}^+(t))$, or to $-J\ast \Gamma(\mathbb{U}^-(t))$,
    and those last convolutions have disjoint supports. Hence we can
    also write $$J\ast \Gamma(\mathbb{U}(t))=J\ast
    \Gamma(\mathbb{U}^+(t))-J\ast \Gamma(\mathbb{U}^-(t))\,.$$ This
    implies that $\mathbb{U}$ is actually a solution of the equation:
    $$\begin{aligned} \partial_t\mathbb{U}&=
      \partial_t\mathbb{U}^+-\partial_t\mathbb{U}^-\\ & =
      J\ast\Gamma(\mathbb{U}^+(t))-\Gamma(\mathbb{U}^+(t))-J\ast\Gamma(\mathbb{U}^-(t))+\Gamma(\mathbb{U}^-(t))\\
      & = J\ast\Gamma(\mathbb{U}(t))-\Gamma(\mathbb{U}(t))\,.
    \end{aligned}$$ But since $\mathbb{U}(0)=f_+-f_-=f$, we conclude by
    uniqueness in $\L^1$ that $u\equiv\mathbb{U}$ is the solution we are
    looking for.  \end{proof}


\section{Asymptotic behaviour when the positive and the negative part of
the temperature do not interact}

The aim of this section is to identify the limit $u_\infty$ (limit of 
the solution $u$ when time goes to infinity) in the case
when the positive and negative part of the temperature, $\Gamma(u)$,
never interact, this is, \begin{equation}\label{cond:supp}
  \dist\left(\suppD\big(\Gamma(\proj{f_+})\big),\,\suppD(\Gamma
  \big(\proj{f_-})\big)\right)\geq R_J\,.  \end{equation}

We know that there exists the retention property for $\mathbb{U}^+$ and
$\mathbb{U}^-$, i.e., the supports of  $\mathbb{U}^+$ and
$\mathbb{U}^-$ are nondecreasing (which holds since these are solutions of the one-phase
Stefan problem). Then we can use the same arguments that have been used
in \cite{BCQ2012}, with the Baiocchi transform, to describe the
asymptotic behaviour of the solution to \eqref{eq1}. For more
information about the Baiocchi transform, (see \cite{Baiocchi}).

On the other hand, we can not say that the solution is
$u(t)=\mathbb{U}^+(t)-\mathbb{U}^-(t)$, like in the example we have
studied in the previous section, because the supports of $\mathbb{U}^+$
and $\mathbb{U}^-$ have an intersection not empty. 

\subsection{Formulation in terms of the Baiocchi variable}

Our next aim is to describe the large time behavior of the
solutions of the two-phase Stefan problem satisfying
hypothesis \eqref{cond:supp}.  We want to make a formulation of the
Stefan problem as a parabolic nonlocal biobstacle problem. To identify
the asymptotic limit for $u$, we define the Baiocchi variable, like in
\cite{BCQ2012}, \[ w(t) =\displaystyle\int_0^t \Gamma(u)(s)\ds.  \] The
enthalpy and the temperature can be recovered from $w$ through the
formulas \begin{equation}\label{igualdad_u_en_funcion_de_w} u = f + J
  \ast w - w, \qquad \Gamma(u)= w_t, \end{equation} where the time
  derivative has to be understood in the sense of distributions. 

  \begin{lema} \label{retention_of_gamma} Under assumption
    \eqref{cond:supp}, the function $\Gamma(u)$ satisfies the following
      {\bf retention property}: for any $0<s<t$,
    \begin{equation}\label{retention_Gamma_u_+_-}\suppD\big(\Gamma(u(s))_+\big)\subset\suppD\big(\Gamma(u(t))_+\big)\,,\quad
    \suppD\big(\Gamma(u(s))_-\big)\subset\suppD\big(\Gamma(u(t))_-\big)\,.
    \end{equation}
    As a consequence, we have for any $t>0$:
    $$\suppD\big(\Gamma(u(t))_+\big)=\suppD\big(w(t)_+\big)\,,\quad
    \suppD\big(\Gamma(u(t))_-\big)=\suppD\big(w(t)_-\big)\,.$$
  \end{lema}

\begin{proof} We use the same ideas as in the previous section. By
  Lemma~\ref{lem:comp.upm} and the retention property
  \ref{retention_property_one_phase} for $\Gamma(\mathbb{U}^+)$ and
  $\Gamma(\mathbb{U}^-)$, we know that for any $t>0$, there holds:
  $$\dist\Big(\suppD\big(\Gamma(u(t))_+\big);
  \suppD\big(\Gamma(u(t))_-\big)\Big)\geq
  \dist\Big(\suppD\big(\Gamma(\proj{f_+})\big);
  \suppD(\Gamma\big(\proj{f_-})\big)\Big)\,,$$ and this distance is at
  least $R_J$ under assumption \eqref{cond:supp}.  Take now a
  nonnegative test function $\phi\in \C^\infty(\R^N)$ (not identically
  zero) with compact support in $\suppD\big(\Gamma(u(s))_+\big)$ and
  consider $t>s$.  Using that $\partial_t\Gamma(u)_+ =
  \ind{\{u>0\}}\partial_t u$, in the sense of distributions, we get $$
  \frac{\mathrm{d}}{\dt}\Big(\int_{\R^N}\Gamma(u(t))_+\phi\Big)=
  \int_{\R^N}\big(J\ast\Gamma(u(t))\big)\phi\ind{\{u>0\}}
  -\int_{\R^N}\Gamma(u(t))\phi\ind{\{u>0\}}\,.$$ Since for any $t>0$,
  the support of $\Gamma(u(t))_+$ is at least at distance $R_J$ from the
  support of $\Gamma(u(t))_-$, we have
  $\big(J\ast\Gamma(u(t))\big)\ind{\{u>0\}}
  =\big(J\ast\Gamma(u(t))_+\big)\geq0$ for any $t>s$. Hence
  $$\frac{\mathrm{d}}{\dt}\Big(\int_{\R^N}\Gamma(u(t))_+\phi\Big)\geq
  -\int_{\R^N}\Gamma(u(t))_+\phi\,,$$ which can be written as $h'(t)\geq
  -h(t)$ where $h(t):=\int_{\R^N} \Gamma(u)(t)_+\phi $.  Hence
    $h(t)\geq h(s)e^{-(t-s)}>0$ which proves the retention
  property for $\Gamma(u)_+$. The property for  $\Gamma(u)_-$ is proved
  the same way.
	
  Now, take a nonnegative test function $\phi$, not identically zero,
  with compact support in $\suppD(\Gamma(u(t))_+)$.  We know from the
  first part that for $0<s<t$, the support of $\phi$ never intersects
  the support of the negative part of $\Gamma(u(s))$, hence
  $$\int_{\R^N}w(t)\phi=\int_0^t\int_{\R^N}\Gamma(u(s))\phi\dx\ds=
  \int_0^t\int_{\R^N}\Gamma(u(s))_+\phi\dx\ds\geq0\,.$$ Moreover, since
  the space integrals are continuous in time, we know that the integral
  $\int_{\R^N}\Gamma(u(s))_+\phi\dx$ is not only positive at time $t$,
  but also in an   open time interval around $t$. So, we get
  $\int_{\R^N}w(t)\phi>0$ which proves that  
  $\suppD\big(\Gamma(u(t))_+\big)\subset\suppD\big(w(t)_+\big)$. On the
  other hand, if $\phi$ is a nonnegative test function such that
  $\int_{\R^N} \Gamma(u(t))_+\phi\dx=0$, the retention property, 
  \eqref{retention_Gamma_u_+_-},  implies
  that this integral is also zero for all times $0<s<t$, which yields
  $\int_{\R^N}w_+(t)\phi\dx=0$. We conclude that the distributional
  support of $w_+(t)$ coincides with that of $\Gamma(u(t))_+$. The proof
  is similar for the negative part.  \end{proof}

The Baiocchi variable satisfies a complementary problem, that will be useful to
introduce the nonlocal biobstacle problem. 

\begin{lema}Under hypothesis \eqref{cond:supp}, the Baiocchi variable, 
$w(t) =\displaystyle\int_0^t \Gamma(u)(s)\ds$,
  satifies the complementary problem almost everywhere
  \begin{equation}\label{desigualdades_complementary_problem}
    \begin{cases} & 0 \le \sign(w)\left(f + J \ast w - w - w_t\right)
      \le 1\,,\\ & \big(f + J \ast w - w- w_t- \sign(w) \big)|w|=0\,,\\
      & w(0) = 0\,.  \end{cases} \end{equation} \end{lema} 

\begin{proof} The graph condition $\Gamma(u)=\sign(u)(|u|-1)_+$ can be
  written as \[ 0 \le \sign(u)\big(u - \Gamma(u)\big) \le 1,\quad
    \left(\sign(u)\big(u - \Gamma(u)\big)-1\right)\Gamma(u) = 0\,, \]
    almost everywhere in $\R^N\times(0,\infty)$\,.  In order to
    translate this condition in the $w$ variable, we first notice that
    that if $\sign\big(\Gamma(u)\big)>0$ then $\sign(u)>0$ and
    similarly, $\sign\big(\Gamma(u)\big)<0$ implies $\sign(u)<0$ (only
    the condition $\Gamma(u)=0$ does not imply a sign condition on $u$).
    Hence we can also write \[ 0 \le \sign\big(\Gamma(u)\big)\big(u -
      \Gamma(u)\big) \le 1,\quad \left(\sign\big(\Gamma(u)\big)\big(u -
      \Gamma(u)\big)-1\right)\Gamma(u) = 0\,.  \] Now we use the
      retention property of $\Gamma(u)$, Lemma \ref{retention_of_gamma},
      which implies that the distributional supports of $\Gamma(u)$ and
      $w$ coincide for all times. Then replacing everything in terms
      of $w$, in \eqref{igualdad_u_en_funcion_de_w}, we have \[
	\begin{cases} 0 \le \sign(w)\left(f + J \ast w - w - w_t\right)
	  \le 1\,,\\ \left(\sign(w)\left(f + J \ast w - w- w_t\right) -
	  1 \right)w=0.  \end{cases} \] Therefore, we obtain that $w$
	  solves a.e. the complementary problem
	  \eqref{desigualdades_complementary_problem}.  \end{proof}

\subsection{A non-local elliptic biobstacle problem}

If
$\displaystyle\int_0^{\infty}\|\Gamma(u)(t)\|_{\L^1{\mathbb{R}^N}}\dt<\infty$,
the function $w(t)$ converges monotonically in
$\L^1\left(\mathbb{R}^N\right)$ as $t\to \infty$ to \[ w_{\infty}
  =\int_0^{\infty} \Gamma(u)(s) \ds \in  \L^1\left(\mathbb{R}^N\right).
  \] Thus, thanks to \eqref{igualdad_u_en_funcion_de_w}, $u(\cdot, t)$
  converges point-wisely and in $\L^1\left(\mathbb{R}^N\right)$ to \[
    \tilde{f} = f + J \ast w_{\infty} - w_{\infty}.  \] Passing to the
    limit as $t \to \infty$ in
    \eqref{desigualdades_complementary_problem}, we get that
    $w_{\infty}$ is a solution with data $f$ to the {\it nonlocal
    biobstacle problem}: \[ \mathrm{(BOP)}\quad \begin{cases}
	\text{Given a data } f \in \L^1\left(\mathbb{R}^N\right), \text{
	find a function } w \in \L^1\left(\mathbb{R}^N\right) \text{
	such that }\\ 0\le \sign(w)\left(f + J \ast w - w\right) \le
	1\,,\\ \big( f + J \ast w - w-\sign(w) \big)|w|=0.  \end{cases}
	\] This problem is called ``biobstacle" since the values of the
	solution are cut at both levels $+1$ and $-1$.  Under some
	conditions we have existence: \begin{lema}\label{lem:cond.ex}
	  Let $f\in\L^1(\R^N)$ satisfy the hypothesis \eqref{cond:supp}.
	  If $N=1$ or $N=2$, assume moreover that J is non increasing in
	  the radial variable, and $f_+\le g_1$, $f_-\le g_2$ for some
	  $g_1,\, g_2\in\L^1\left(\mathbb{R}^N\right)\cap
	  \C_0(\mathbb{R}^N)$, radial and strictly decreasing in the
	  radial variable.  Then, problem $\mathrm{(BOP)}$ has at least
	  a solution $w_\infty\in\L^1(\R^N)$.  \end{lema} \begin{proof}
	    Given the assumptions, we construct the solution $u$ of
	    \eqref{eq1} associated to the initial data $f$. Then we use
	    the estimate
	    $$|\Gamma(u)|\leq\max\big((\mathbb{U}^+-1)_+\,;\,(\mathbb{U}^--1)_+\big)\,.$$
	    If $N\geq3$, we use  \cite[Cor. 3.11]{BCQ2012} to get
	    $\|\Gamma(u(t))\|_{\L^1(\R^N)}=O(t^{-N/2})$.  For dimensions
	    $N=1,2$, we use the extra assumption and \cite[Cor.
	    3.10]{BCQ2012} which implies
	    $\|\Gamma(u(t))\|_{\L^1(\R^N)}\leq C e^{-\kappa t}$ for some
	    $C,\kappa>0$.  In both cases, we obtain that $\int_0^\infty
	    \Gamma(u(s))\ds$ converges in $\L^1(\R^N)$ to some function
	    $w_\infty$, and passing to the limit in
	    \eqref{desigualdades_complementary_problem} we see that
	    $w_\infty$ is a solution of (BOP).  \end{proof}

We now have a more general uniqueness result (without extra assumptions
in lower dimensions).

\begin{prop} Given any function $f\in\L^1(\R^N)$, the problem $\mathrm{(BOP)}$
has at most one solution $w\in\L^1(\R^N)$.
\end{prop}
  \begin{proof}
   The proof follows the same arguments as in \cite[Thm
    5.3]{BCQ2012}.  For the sake of completeness we reproduce here the
    argument:
	the solutions of (BOP) satisfy, $$\tilde{f}=f+J\ast w -w\,,\quad
	\tilde{f}\in\beta(w)\ \text{a.e.}\,,$$ where $\beta(\cdot)$ is
	the graph of the sign function: $\beta(w)=\sign(w)$ if $w\neq0$,
	and   $\beta(\{0\})=[-1,1]$.  We take two solutions
	$w_i$, $i=1,2$ of (BOP) associated with the data $f$ and let
	$\tilde{f}_i$ be the associated projections. Since
	$\tilde{f}_i\in\beta(w_i)$ we have $$ 0\le(\tilde f_1-\tilde
	f_2)\ind{\{w_1>w_2\}}=
	\big(J*(w_1-w_2)-(w_1-w_2)\big)\ind{\{w_1>w_2\}}\quad\text{a.e.}\,.$$
	We then use a nonlocal version of Kato's inequality, valid for
	locally integrable functions: \begin{equation*}
	  (J*w-w)\ind{\{w>0\}}\le J*w_+-w_+\quad\text{a.e.},
	\end{equation*}	which implies $$(w_1-w_2)_+\le
	J*(w_1-w_2)_+\,.$$ We end by using \cite[Lem 6.2]{BCQ2012}, from
	which we infer that $(w_1-w_2)_+=0$. Reversing the roles of
	$w_1$ and $w_2$ we get uniqueness.  \end{proof}

Combining the results above, we can now give our main theorem concerning
the asymptotic behaviour for solutions of \eqref{eq1}.

\begin{teo}\label{mesa_problem} Let $f\in\L^1\left(\mathbb{R}^N\right)$,
  satisfying the assumptions of Lemma~\ref{lem:cond.ex}, if $N=1$ or
  $2$.  If $u$ is the unique solution to the problem \eqref{eq1} and
  $w_\infty$ is the unique solution of the problem $\mathrm{(BOP)}$, we
  have $$u(t)\to \tilde{f}:=f+J\ast w_\infty-w_\infty\quad \text{in }
  \L^1\left(\mathbb{R}^N\right)\quad\text{as }t\to\infty\,.$$ \end{teo}

\subsection{Asymptotic limit for general data} 

Up to now we have been able to prove the existence of a solution of
(BOP) for any $f \in \L^1\left(\mathbb{R}^N\right)$ only if $N \ge 3$.
For low dimensions, $N = 1,\, 2$, we have needed to add the hypotheses
of Lemma~\ref{lem:cond.ex}.  Hence, for lower dimensions the projection
operator $\mathcal{P}$ which maps $f$ to $\tilde{f}$ is in principle
only defined under these extra assumptions.

However, $\mathcal{P}$ is continuous, in the $\L^1$-norm, in the subset
of $\L^1\left(\mathbb{R}^N\right)$ of functions satisfying the
hypotheses of Lemma~\ref{lem:cond.ex}.  Since the class of functions
satisfying those hypotheses is dense in $\L^1\left(\mathbb{R}^N\right)$,
we can extend the operator to all $\L^1$ by a standard procedure.

\begin{teo} Let $f \in \L^1\left(\mathbb{R}^N\right)$ and $u$ the
  corresponding solution to problem \eqref{eq1}.  Let $\mathcal{P}f$ be
  the projection of $f$ onto $\tilde{f}$.  Then
  $u(\cdot, t) \to \mathcal{P}f$ in $\L^1\left(\mathbb{R}^N\right)$ as
  $t \to\infty$.  \end{teo} 

\begin{proof} Given $f$, let $\{f_n\} \subset
  \L^1\left(\mathbb{R}^N\right)$ be a sequence of functions satisfying
  the hypotheses of Lemma~\ref{lem:cond.ex} which approximate $f$ in
  $\L^1\left(\mathbb{R}^N\right)$. Take for instance a sequence of
  compactly supported functions. Let $u_n$ be the corresponding
  solutions to the non-local Stefan problem. We have, \[
    \|u(t)-\mathcal{P}f\|_{\L^1\left(\mathbb{R}^N\right)} \le
    \|u(t)-u_n(t)\|_{\L^1\left(\mathbb{R}^N\right)}+
    \|u_n(t)-\mathcal{P}f_n\|_{\L^1\left(\mathbb{R}^N\right)}+
    \|\mathcal{P}f_n-\mathcal{P}f\|_{\L^1\left(\mathbb{R}^N\right)}.  \]
    Using  Corollary \ref{contract_prop_L_1}, which gives
    the contraction property for the 
    non-local Stefan problem, and Theorem \ref{mesa_problem}, that states
    the large time behavior for bounded and compactly supported initial data,  
    we obtain
     \[ \limsup\limits_{t\to \infty} \|u(t) -
      \mathcal{P}f\|_{\L^1\left(\mathbb{R}^N\right)}\le \|f -
      f_n\|_{\L^1\left(\mathbb{R}^N\right)} + \|\mathcal{P}f_n -
      \mathcal{P}f\|_{\L^1\left(\mathbb{R}^N\right)}.  \] Letting $n
      \to\infty$ we get the result.  \end{proof}

\begin{remark} A similar result would be valid for the local Stefan
  problem, assuming that the distance between $\mathcal{P}f_+$ and
  $\mathcal{P}f_-$ is strictly positive. Notice that the projected data 
  $\tilde{f}$ is a non-local mesa, see \cite{BCQ2012}.
 \end{remark}

\section{Solutions losing one phase in finite time.}

In this section we we give some partial results on the asymptotic
behaviour of solutions for which either $u$ or $\Gamma(u)$ becomes
nonnegative (or nonpositive) in finite time. 

In this case, we can prove that the asymptotic behaviour is driven  by
the one-phase Stefan regime, however we cannot identify the limit
exactly.


\subsection{A theoretical result.}

\begin{teo}\label{thm:gamma.positive} Let $f\in\L^1(\R^N)$ satisfy
  \eqref{cond:supp} and let $u$ be the corresponding solution. Assume
  that for some $t_0\geq0$, there holds $f^*:=u(t_0)\geq -1$ in $\R^N$.
  then the asymptotic behaviour is given by: $u(t)\to\mathcal{P}f^*$.
\end{teo}

\begin{proof} We just have to consider $u^*(t):=u(t-t_0)$ for $t\geq
  t_0$. Then $u^*$ is the solution associated to the initial data $f^*$
  which satisfies \eqref{cond:supp}. Hence we know that as $t\to\infty$,
  $u^*(t)\to\mathcal{P}f^*$. Therefore, the same happens
 for $u(t)$.  
\end{proof}

Of course a similar result holds if $\Gamma(u)$ becomes nonpositive in
finite time. However, the problem remains open as to identify
$\mathcal{P}f^*$ since we do not know what is exactly $f^*$.

In the rest of the section, we give two examples where such a phenomenon
occurs.  One for which $v=\Gamma(u)$ becomes positive in finite time, and the
other for which $u$ becomes positive in finite time.

\subsection{Sufficient conditions to lie above level $-1$ in finite
time.}
 
In this subsection we assume for simplicity that the initial data $f$ is
continuous and compactly supported, and that $J$ is nonincreasing in the
radial variable.  
We assume $f_+\le g_1$ and $f_-\le g_2$, for some $g_1,\,  g_2\in
\L^1\left(\mathbb{R}^N\right)\cap \C_0(\mathbb{R}^N)$ radial and
strictly decreasing in the radial variable. Moreover,  
Thanks to \cite[Lem 3.9]{BCQ2012}, there exists $R=R(g_1,g_2)$ such
that $\supp\big(v(u)(t)\big)\subset B_R$ for any $t\geq0$ (recall that
we denote by $v=\Gamma(u)$). Notice
that $R$ does not depend on $J$, only on the $\L^1$-norm of $g_1$ and
$g_2$.

We make first the following important assumption:
\begin{equation}\label{assumption:alpha}
  \alpha(v_0,J):=
  \inf_{x\in B_R}\int J(x-y)v_+(y,0)dy>0
\end{equation}
(see in the Remark \ref{remark_hyp_12} below 
some comments on this assumption).\\
Let us also denote $$\beta(J):=\sup_{x\in B_{2R}} J(x).$$
Then we shall also assume that the negative part of $v_0:=v(0)$ 
is ``small'' compared to the
positive part in the following sense:
\begin{equation}\label{assumption:neg.pos}
  \|v_-(0)\|_{\L^1(\R^N)}<\frac{\alpha(v_0,J)}{\beta(J)}\,.
\end{equation}
In such a situation, we first define 
\begin{equation*}
  \bar\eta:=\alpha(v_0,J)-\beta(J)\|v_-(0)\|_{\L^1(\R^N)}>0\,.
\end{equation*}
Then, for $\eta\in(0,\bar\eta)$ we introduce the following function
\begin{equation*}
  \varphi(\eta):=\eta\ln\left(
  \frac{\alpha(v_0,J)}{\eta+
    \beta(J)\|v_-(0)\|_{\L^1(\R^N)}}\right)>0
\end{equation*}
and set
$$\kappa:=\max\big\{\varphi(\eta):\eta\in(0,\bar\eta)\big\}>0.$$
Since actually, $\kappa$ depends only on $J$ and the mass of the 
positive and negative parts of $v(0)$, we denote it by $\kappa(v_0,J)$.

We are then ready to formulate our result:
\begin{prop}
  Assume \eqref{assumption:neg.pos} and moreover
   that the negative part of $f$ is controlled in the sup norm as follows 
   \begin{equation*}
       \|f_-\|_\infty\le 1+\kappa(v_0,J)\,.
   \end{equation*}
   Then in a finite time $t_1=t_1(f)$, 
   the solution satisfies $u(x,t_1)\geq -1$ for all $x\in\R^N$.
\end{prop} 

\begin{proof} By our assumptions, for all $x$ we have 
  $f(x)\geq -1-\kappa(v_0,J)$. Then for any $x\in B_R$,
  \[ \begin{array}{ll} J\ast
    v(x,0) &
    =\displaystyle\int_{\{v>0\}}J(x-y)v(y,0)dy+\int_{\{v<0\}}J(x-y)v(y,0)dy
    \smallskip\\ & \ge \alpha(v_0,J)
    -\beta(J)\|v_-(0)\|_{\L^1(\R^N)}>0.
  \end{array} \] 
Remember that for the points $x\notin B_R$, we have $v_0(x)=0$ and also
$v(x,t)=0$ for any time $t\geq0$ (though we may ---and will--- have mushy regions, 
$\{|v|<1\}$, outside $B_R$ of course).

Thanks to the continuity of $u$ (and $v$), the following time is
well-defined:
\begin{equation*}
  t_0:=\sup\{t\geq0: J\ast v(x,t) >0 \text{ for any } x\in B_R\}>0\,.
\end{equation*}
This implies that $$u_t \geq -v, \;\mbox{ in }\; B_R\times(0,t_0),$$ so that
\begin{equation*}
  \partial_t v_+ = \ind{\{v>0\}} \partial_t u\geq -v\ind{\{v>0\}}=-v_+\,. 
\end{equation*}
Hence, in $B_R\times(0,t_0)$, $v_+$ enjoys the following retention property:
\begin{equation}\label{retent} 
  v_+(x,t)\ge e^{-t}v_+(x,0), \;\forall t\in [0,t_0).  
\end{equation} 
This implies in particular that if $v(x,0)$ is positive at some point,
$v(x,t)$ remains positive at this point at least until $t_0$.

Now, let us estimate $t_0$. For any $x\in B_R$ and $t\in(0,t_0)$, we have
\[ \begin{array}{ll} J\ast v(x,t) & \geq 
  \displaystyle\int_{\{v>0\}}J(x-y)v(y,t)dy+\int_{\{v<0\}}J(x-y)v(y,t)dy
  \smallskip\\ & \ge
 e^{-t}\displaystyle\int_{\{v>0\}}J(x-y)v(y,0)dy-
   \beta(J)\|v_-(t)\|_{\L^1\left(\mathbb{R}^N\right)}
   \smallskip\\ & \ge
   \alpha(v_0,J) e^{-t}-
   \beta(J)\|v_-(0)\|_{\L^1(\R^N)}\,,  \end{array} 
 \]
where we have used Corollary \ref{contract_prop_L_1}, which gives the $\L^1$-contraction property for $v_-$, deriving from the fact that it is subcaloric.
So, if we take $\eta$ reaching the max of $\varphi(\eta)=\kappa$ and set 
\begin{equation*}
  t_1(\eta):= \ln \bigg( \dfrac{\alpha(v_0,J)}{ \eta +
	 \beta(J) \|v_-(0)\|_{\L^1(\R^N)}}\bigg)\,,
\end{equation*}
then for any $t\in(0,t_1)$, we have 
$\alpha(v_0,J) e^{-t}-\beta(J) 
\|v_-(0)\|_{\L^1(\R^N)}> \eta>0$.
This proves that $t_0\geq t_1$.

Since $v_+$ has the retention property in $(0,t_0)$, the points in
$$\mathcal{C}^+:=\{x\in\R^N:v(x,0)>0\}$$ remain in this set at least until
$t_0$.

Then, for any $x\in\mathcal{C}^-:=\{x\in\R^N: v(x,0)\leq0\}$,
we define $$t(x):=\sup\{t>0:v(x,t)\leq0\}.$$
If $t(x)=0$, this means that $v(x,t)$ becomes positive immediately and
will remain as such at least until $t_1$ 
so we do not need to consider such points. We are
left with assuming $t(x)>0$ (or infinite).

Then in this last case, we shall prove that $t(x)\leq t_1$ by
contradiction: let us assume that $t(x)>t_1$ and let us come back to the
previous estimate. We then have, for any $t\in(0,t_1)$:
\[  u_t(x,t) = J\ast v(x,t)-v(x,t) \ge J\ast v(x,t)>\eta>0\,.\]
Thus, integrating the equation in time at $x$ yields
\begin{equation*}
  u(x,t)> -1-\kappa(v_0,J)+ \eta\cdot t,\;\;\forall t\in
  [0,t_1].
\end{equation*}
By our choice we have precisely $\kappa(v_0,J)=\varphi(\eta)=\eta\cdot t_1(\eta)$.
Therefore, at least for $t=t_1$, we have
\begin{equation*}
  u(x,t_1)> -1-\kappa(v_0,J)+\eta\cdot t_1 > -1\,,
\end{equation*}
which is a contradiction with the fact that $t(x)>t_1$. Hence $t(x)\leq
t_1$, which means that at such points, the solution becomes equal to or
above level $-1$ before $t_1$.

So, combining everything, we have finally obtained that for any point
$x\in\R^N$, $u(x,t)$ becomes greater than or equal to $-1$ before the
time $t_1$, which ends the proof.  
\end{proof}
	 
\begin{remark}\label{remark_hyp_12} Hypothesis 
\eqref{assumption:alpha} expresses that for
  any $x\in B_R$, there is some positive contribution in the convolution
  with the positive part of $v_0$. So, this implies that at least the
  following condition on the intersection of the supports should hold:
  $$\forall x\in B_R\,,\quad \big(x+B_{R_J}(0)\big)\cap \supp\big( (v_0)_+
  \big)\neq\emptyset\,.$$ Actually, if the radius $R_J$ is big enough to
  contain all the support of $v_0$ this is satisfied. But even if it is
  not so big, it there are positive values of $v_0$ which spread in many
  directions, this condition can be satisfied. 
  
  Then, \eqref{assumption:neg.pos} is a condition on the negative part,
  which should not be too big so that all the possible points such that 
  $v(x,0)<0$ will enter into the positive set for $v$ in finite time. 
  The exact control is a mix
  between the mass and the infinite norm of the various quantities.
\end{remark}


\begin{thebibliography}{99}

\bibitem{Baiocchi} Baiocchi, C. {\it Su un problema di frontiera libera
  conneso a questioni di idraulica.} Ann. Mat. Pura. Appl. (4) 92
  (1972), 107-127.  
  
  \bibitem{BCQ2012} Br\"andle, C.; Chasseigne, E.; Quir\'os,
    F. {\it Phase transition with mid-range interactions: a nonlocal
    one-phase Stefan model}, SIAM J. Math. Anal., Vol. 44, No. 4, 
    (2012) 3071--3100.
    
  \bibitem{BrandleChasseigneFerreira2011} Br\"andle, C.; Chasseigne, E.; Ferreira, R.
  {\it Unbounded solutions of the nonlocal heat equation}. Commun.
  Pure Appl. Anal. 10 (2011), no.~6, 1663--1686.

  
  \bibitem{Chalmers} Chalmers, B. ``Principles of solidification''.
    Wiley, 1964.
  
  \bibitem{Crank}Crank, J. ``Free and moving boundary problems''. The
    Clarendon Press, Oxford University Press, New York, 1987.

  \bibitem{IgnatRossi} Ignat, L.; Rossi, J.~D. {\it Refined asymptotic 
  expansions for nonlocal diffusion equations}. J. Evol. Equ. 8 (2008), no.~4, 617--629.

  
  \bibitem{Lame} Lamé, G.; Clapeyron, B. P. {\it Mémoire sur la
    solidification par refroidissement d'un globe solid.} Ann. Chem.
    Phys. 47 (1831), 250-256.
  
  \bibitem{Meirmanov} Meirmanov, A. M. ``The Stefan problem''. Walter de
    Gruyter, Berlin, 1992.
  
  \bibitem{Rubinstein}Rubinstein, L. I. ``The Stefan problem''.
    Zvaigzne, Riga, 1967 (in Russian). English transl.: Translations of
    Mathematical Monographs, Vol. 27. American Mathematical Society,
    Providence, R.I., 1971.
  
  \bibitem{Stefan}Stefan, J. {\it \"Uber einige Probleme der Theorie
    der W\"armeleitung. Sitzungsber}, Wien, Akad.  Mat. Natur. 98
    (1889), 473-484; see also pp. 614-634; 965-983; 1418-1442.
  
  \bibitem{Visintin}Visintin, A. {\it Two-scale model of phase
    transitions.} Phys. D 106 (1997), no. 1-2, 66-80.
  
  \bibitem{Woodruff}Woodruff, D. P. ``The solid-liquid interface''.
    Cambridge University Press, London 1980.  
  
  \end{thebibliography}
\end{document}